\documentclass[12pt]{amsart}
\usepackage{a4wide}
\usepackage[T1]{fontenc}
\usepackage{amssymb,amsmath,amsthm,latexsym}
\usepackage{mathrsfs}
\usepackage[usenames,dvipsnames]{color}
\usepackage{euscript}
\usepackage{graphicx}
\usepackage{mdwlist}
\usepackage{enumerate}
\usepackage{mathtools,dsfont,wasysym}
\usepackage{bbm}
\usepackage{stmaryrd}
\usepackage{centernot}
\usepackage{todonotes}
\usepackage{hyperref}
\hypersetup{colorlinks=true,  linkcolor=blue, citecolor=magenta, urlcolor=cyan}

\makeatletter
\@namedef{subjclassname@2020}{\textup{2020} Mathematics Subject Classification}
\makeatother

\newtheorem{theorem}{Theorem}[section]
\newtheorem{lemma}[theorem]{Lemma}
\newtheorem{corollary}[theorem]{Corollary}

\newcounter{maintheorem}

\theoremstyle{remark}
\newtheorem{remark}[theorem]{Remark}
\theoremstyle{definition}
\newtheorem{definition}[theorem]{Definition}

\newtheorem{example}[theorem]{Example}
\numberwithin{equation}{section}
\makeatother

\newcommand{\R}{\mathbb{R}}
\newcommand{\N}{\mathbb{N}}
\newcommand{\e}{\varepsilon}
\newcommand{\p}{\varphi}
\newcommand{\si}{\frac{\sigma+1}{\sigma-1}}
\newcommand{\sii}{\frac{2}{\sigma-1}}
\newcommand{\siii}{\frac{2\sigma}{\sigma-1}}

\newcommand{\nn}[1]{{\left\vert\kern-0.25ex\left\vert\kern-0.25ex\left\vert #1 
\right\vert\kern-0.25ex\right\vert\kern-0.25ex\right\vert}}

\renewcommand{\leq}{\leqslant}
\renewcommand{\geq}{\geqslant}
\newcommand{\dif}{\nabla_{xy}}

\DeclareMathOperator{\Vol}{Vol}

\newcounter{smallromans}

\newenvironment{romanenumerate}
{\begin{list}{{\normalfont\textrm{(\roman{smallromans})}}}%
  {\usecounter{smallromans}\setlength{\itemindent}{0cm}%
   \setlength{\leftmargin}{5.5ex}\setlength{\labelwidth}{5.5ex}%
   \setlength{\topsep}{.5ex}\setlength{\partopsep}{.5ex}%
   \setlength{\itemsep}{0.1ex}}}%
{\end{list}}

\newcommand{\bone}{\mathbbm{1}}
 
\begin{document}
\title[Nonexistence results on graphs]{Nonexistence results for semilinear elliptic \\ equations on weighted graphs}

\author[D.D.~Monticelli]{Dario D. Monticelli}
\address[D.D.~Monticelli]{Politecnico di Milano, Dipartimento di Matematica, Piazza Leonardo da Vinci 32, 20133 Milano, Italy}
\email{dario.monticelli@polimi.it}

\author[F.~Punzo]{Fabio Punzo}
\address[F.~Punzo]{Politecnico di Milano, Dipartimento di Matematica, Piazza Leonardo da Vinci 32, 20133 Milano, Italy}
\email{fabio.punzo@polimi.it}

\author[J.~Somaglia]{Jacopo Somaglia}
\address[J.~Somaglia]{Politecnico di Milano, Dipartimento di Matematica, Piazza Leonardo da Vinci 32, 20133 Milano, Italy}
\email{jacopo.somaglia@polimi.it}

\thanks{The research of the authors has been partially supported by GNAMPA (INdAM -- Istituto Nazionale di Alta Matematica).}

\date{\today}

\begin{abstract} We study semilinear elliptic inequalities with a potential on infinite graphs. Given a distance on the graph, we assume an upper bound on its Laplacian, and a growth condition on a suitable weighted volume of balls. Under such hypotheses, we prove that the problem does not admit any nonnegative nontrivial solution. We also show that our conditions are optimal.   
\end{abstract}
\maketitle

%-------------------------------------------------------%
%                                                       %
% 						INTRODUCTION 					%
%                                                       %
%-------------------------------------------------------%

\section{Introduction}
One of the most interesting and investigated class of elliptic
inequalities, in view of its relevance in theoretical questions and applications, is
\begin{equation}\label{EQ_lapl}
 \Delta u+v(x)u^\sigma \leq 0,
\end{equation}
both on $\mathbb{R}^N$ and on general Riemannian manifolds $(M,g)$,
where $\Delta$ denotes the Laplace-Beltrami operator associated to
the metric and $\sigma>1$. 

The study of this problem has a very rich history, starting with the seminal works of Gidas
\cite{Gidas} and Gidas-Spruck \cite{GidasSpruck} in the Euclidean setting. It is shown  that $u\equiv 0$ is the only nonnegative
solution of equation \eqref{EQ_lapl} whenever $1<\sigma\leq \frac{N}{N-2}$, in case $v\equiv 1$ and $N\geq 3$. We refer to the fundamental papers of D'Ambrosio-Mitidieri \cite{DaMit} and Mitidieri-Pohozaev
\cite{MitPohoz359}, \cite{MitPohoz227}, \cite{MitPohozAbsence},
\cite{MitPohozMilan} for a comprehensive account of results
related to these (and also more general) problems on $\mathbb{R}^N$. 
Note that analogous results have also been obtained for degenerate
elliptic equations and inequalities (see, e.g., \cite{DaLu},
\cite{Mont}) and for the parabolic companion problems (see, e.g.,
\cite{MitPohozMilan}, \cite{PoTe}, \cite{P1}, \cite{PuTe}). 

By means of capacity tools, which make use of carefully chosen test
functions, Mitidieri-Pohozaev prove nonexistence of weak or
distributional solutions  of wide classes of differential
inequalities in $\mathbb{R}^N$. In particular, they show that equation
\eqref{EQ_lapl} on $\mathbb{R}^N$ does not admit any nontrivial
nonnegative solution, provided that
\begin{equation*}
\liminf_{R\rightarrow+\infty}R^{-\frac{2\sigma}{\sigma-1}}\int_{B_{\sqrt{2}R}\setminus
B_R}v(x)^{-\frac{1}{\sigma-1}}dx<\infty.
\end{equation*}

More recently, problem \eqref{EQ_lapl} has been investigated also in the Riemannian setting; in particular we cite the papers 
\cite{GrigKond}, 
\cite{GrigS}, \cite{Kurta}, \cite{MMP}, \cite{MMP2} \cite{Sun1}. Using a capacity argument which only exploits the
gradient of the distance function from a fixed reference point,
%and assumptions on a suitably weighted volume growth of geodesic balls,
in particular the authors showed in \cite{GrigKond} that equation
\eqref{EQ_lapl} admits a unique nonnegative weak
solution provided that there exist positive constants $C$, $C_0$
such that for every $R>0$ sufficiently large and every small enough
$\e>0$,
\begin{equation}\label{Eq_noimp}
\int_{B_R(x_0)}v^{-\beta+\e}\,d\mu \leq C
R^{\alpha+C_0\e}(\log R)^k,
\end{equation}
where $d\mu$ is the canonical Riemannian measure on $M$, $B_R(x_0)$ is
the geodesic ball centered at a point $x_0\in M$ and
\begin{equation*}
\alpha = \frac{2\sigma}{\sigma-1}, \quad \beta = \frac{1}{\sigma-1}\,,
\end{equation*}
\[
0 \leq k <\beta.
\]
The sharpness of the
exponent $\alpha$ in this type of results is evident from the
Euclidean case \eqref{EQ_lapl} with $v\equiv 1$ and $N\geq 3$,
where $\alpha=N$ and the corresponding critical growth is $\sigma^*
=\frac{N}{N-2}$. The sharpness of the exponent $\beta$ is
definitely a more delicate
 question and has recently been settled on a general Riemannian
manifold, in case $v\equiv 1$, in \cite{GrigS}. In
that paper, using carefully chosen families of test functions, the
authors showed that equation \eqref{EQ_lapl} with $v\equiv 1$ does
not admit any nonnegative weak solution provided \eqref{Eq_noimp}
holds with $k=\beta$. 

Observe that on a general Riemannian manifold only the gradient of the distance function from a fixed point is used, and not its Laplacian, since the former is always well-defined in the weak sense, while the latter may not exist due to the possible presence of a cut-locus. Therefore, with this type of arguments, while on $\mathbb{R}^N$ one may consider very weak solutions, on a general Riemannian manifold one is forced to deal with weak solutions.

\smallskip

The aim of this paper is to study problem \eqref{EQ_lapl} on weighted graphs 
$(V,\omega,\mu)$, where $\omega$ and $\mu$ are the so called \textit{edge weight} and \textit{node measure}, respectively (see Section \ref{mb} for more details).
Recently, the study of elliptic and parabolic equations on graphs has attracted much attention (see, e.g., \cite{Grigbook}, \cite{GLY}, \cite{GLY2}, \cite{HM}, \cite{HKMW}, \cite{HKS}, \cite{KS}, \cite{KLW}, \cite{LW}, \cite{PS}). 

In particular, in \cite{BMP} and \cite{MP}  Liouville type theorems for equation \eqref{EQ_lapl} are obtained on weighted graphs 
$(V,\omega,\mu)$, when $\sigma=1$. 
In \cite{GHS}, for $\sigma>1$, under the assumption that
there exists $p_0>1$ such that 
\begin{equation}\label{e2p}
\frac{\omega(x,y)}{\mu(x)}\geq \frac 1{p_0}
\end{equation}
for all $x, y\in V$ which are endpoints of an edge of $V$,
it is shown that equation
\begin{equation}\label{e1p}
\Delta u + v(x) u^\sigma\leq 0 \quad \text{ in } V
\end{equation}
with $v\equiv 1$ does not admit any nontrivial nonnegative solutions, if, for some $x_0\in V, C>0$ 
\[\operatorname{Vol}(B_R(x_0)):=\sum_{x\in B_{R}(x_0)} \mu(x)\leq CR^{\frac{2\sigma}{\sigma-1}}\log R^{\frac 1{\sigma-1}}\quad \text{ for all }
     R\geq R_0\,,
\]
where $B_R(x_0)\subset V$ is the ball with centre $x_0\in V$ and radius $R>0$, defined in terms of the natural distance on $V$. The technique used in \cite{GHS} relies on a test function argument inspired from \cite{GrigS} and is based on a single integration by parts.

In this paper we investigate nonexistence of nonnegative nontrivial solutions to equation \eqref{e1p}. With respect to the results established in \cite{GHS}, we remove assumption \eqref{e2p}, we allow for the presence of a potential $v$, and we can consider a general pseudo-metric $d$ on the graph rather than using the natural distance on the graph. 
We assume an upper bound on the Laplacian of $d(\cdot, x_0)$, for a fixed reference point $x_0\in V$, together with a related growth condition on the weighted volume of annuli centered at $x_0$. 
Under such hypotheses we show that the only nonnegative solution $u$ to equation \eqref{e1p} is $u\equiv 0$, see Theorem \ref{t: mainthm}. Note that this kind of hypothesis on the Laplacian of the distance arises quite naturally when studying problems on Riemannian manifolds, and it follows from curvature conditions on the manifold (see also Remark \ref{curvatura}).

In our proofs we exploit a test functions argument with test functions depending on $d(\cdot, x_0)$. 
We integrate by parts twice, since on graphs no regularity issues arise for functions. 
Our results are sharp; we discuss this aspect with a few examples: in one we deal with the standard lattice $\mathbb Z^N$ and in another one we consider a homogeneous tree.
Let us mention that we construct one example for which the nonexistence results in \cite{GHS} cannot be used, since hypothesis \eqref{e2p} fails, while our main nonexistence theorem applies. 

\smallskip

The paper is organized as follows. In Section \ref{mb} we introduce basic notions on graphs. In Section \ref{mr} we state our main nonexistence theorem, which is proved in Section \ref{proofmt}. Finally, in Section \ref{ex} we provide some examples in which our main theorem applies, while in Section \ref{optimal} we show optimality of our result.

\section{Mathematlcal background}\label{mb}
\subsection{The graph setting}
Let $V$ be a countably infinite set and $\mu:V\to (0,+\infty)$ be a given function. %Observe that $\mu$ can be viewed as a Radon measure on $V$ so that $(V,\mu)$ becomes a measure space. 
Furthermore, let
\begin{equation*}
\omega:V\times V\to [0,+\infty)
\end{equation*}
be a symmetric function with zero diagonal and finite sum, i.e.
\begin{equation}\label{omega}
\begin{array}{ll}
\text{(i)}\,\, \displaystyle\omega_{xy}=\omega_{yx}&\text{for all}\,\,\, (x,y)\in V\times V;\\
\text{(ii)}\,\, \displaystyle\omega_{xx}=0 \quad\quad\quad\,\, &\text{for all}\,\,\, x\in V;\\
\text{(iii)}\,\, \displaystyle \sum_{y\in V} \omega_{xy}<\infty \quad &\text{for all}\,\,\, x\in V\,.
\end{array}
\end{equation}
Thus, we define  \textit{weighted graph} the triplet $(V,\omega,\mu)$, where $\omega$ and $\mu$ are the so called \textit{edge weight} and \textit{node measure}, respectively. Observe that assumption (ii) corresponds to ask that $V$ has no loops. 
\smallskip

\noindent Let $x,y$ be two points in $V$; we say that
\begin{itemize}
\item $x$ is {\it connected} to $y$ and we write $x\sim y$, whenever $\omega_{xy}>0$;
\item the couple $(x,y)$ is an {\it edge} of the graph and the vertices $x,y$ are called the {\it endpoints} of the edge whenever $x\sim y$; we will denote by $E$ the set of all edges of the graph;
\item a collection of vertices $ \{x_k\}_{k=0}^n\subset V$ is a {\it path} if $x_k\sim x_{k+1}$ for all $k=0, \ldots, n-1.$
\end{itemize}

A weighted graph $(V,\omega,\mu)$  is said to be 
\begin{itemize}
\item[(i)] {\em locally finite} if each vertex $x\in V$ has only finitely many $y\in V$ such that $x\sim y$;
\item[(ii)] {\em connected} if, for any two distinct vertices $x,y\in V$, there exists a path joining $x$ to $y$;
\item[(iii)] {\em undirected} if its edges do not have an orientation.
\end{itemize}
For any $x\in V$ we will call {\em degree of $x$} the cardinality of the set $$\{y\in V\colon y\sim x\}.$$ 

\noindent A {\it pseudo metric} on $V$ is a symmetric map, with zero diagonal, $d:V\times V\to [0, +\infty)$, which also satisfies the triangle inequality
\begin{equation*}
 % \item $d(x,x)=0$ \quad \text{ for all } $x\in V$;
  %\item $d(x,y)=d(y,x)$ \text{ for all } $x,y\in V$;
  d(x,y)\leq d(x,z)+d(z,y)\quad \text{for all}\,\,\, x,y,z\in V.
\end{equation*}
In general, $d$ is not a metric, since there may be points $x, y\in V$, $x\neq y$ such that $d(x,y)=0\,.$ 
Finally, we define the \textit{jump size} $j>0$ of a general pseudo metric $d$ as
\begin{equation}\label{e14f}
j:=\sup\{d(x,y) \,:\, x,y\in V, \omega(x,y)>0\}.
\end{equation}

For any given $\Omega\subset V$ its {\it volume} is defined as
\[\operatorname{Vol}(\Omega):=\sum_{x\in \Omega}\mu(x)\,.\]

\subsection{The weighted Laplacian} Let $\mathfrak F$ denote the set of all functions $f: V\to \mathbb R$\,. For any $f\in \mathfrak F$ and for all $x,y\in V$, let us give the following
\begin{definition}\label{def1}
Let $(V, \omega,\mu)$ be a weighted graph. For any $f\in \mathfrak F$,
\begin{itemize}
\item the {\em difference operator} is
\begin{equation}\label{e2f}
\nabla_{xy} f:= f(y)-f(x)\,;
\end{equation}
\item the {\em (weighted) Laplace operator} on $(V, \omega, \mu)$ is
\begin{equation*}
\Delta f(x):=\frac{1}{\mu(x)}\sum_{y\in V}\omega_{xy}[f(y)-f(x)]=\frac 1{\mu(x)}\sum_{y\sim x}\omega_{xy}[f(y)-f(x)]\quad \text{ for all }\, x\in V\,.
\end{equation*}
\end{itemize}
\end{definition}
Clearly,
\[\Delta f(x)=\frac 1{\mu(x)}\sum_{y\in V}\omega_{xy}(\dif f)\quad \text{ for all } x\in V\,.\]
We also define  the {\it gradient squared} of $f\in \mathfrak F$ (see \cite{CGZ})
\[|\nabla f(x)|^2=\frac 1{\mu(x)}\sum_{y\in V}\omega_{xy}(\dif f)^2, \quad x\in V\,;\]
It is straightforward to show, for any $f,g\in \mathfrak F$, the validity of
\begin{itemize}
\item  the {\it product rule}
\begin{equation*}
\nabla_{xy}(fg)=f(x) (\nabla_{xy} g) + (\nabla_{xy} f)g(y) \quad \text{ for all } x,y\in V\,;
\end{equation*}
\item the {\it integration by parts formula}
\begin{equation}\label{e4f}
\sum_{x\in V}[\Delta f(x)] g(x) \mu(x)=-\frac 1 2\sum_{x,y\in V}\omega_{xy}(\dif f)(\dif g)= \sum_{x\in V}f(x) [\Delta g(x)] \mu(x)\,,
\end{equation}
provided that at least one of the functions $f, g\in \mathfrak F$ has {\it finite} support.
\end{itemize}

For any $x_0\in V$ and $r>0$ we define the ball $B_r(x_0)$ with respect to any pseudo metric $d$ as
\[B_r(x_0):=\{x\in V\,:\, d(x,x_0)\leq r\}\,.\]

\subsection{Assumptions.} In this paper, we always make the following assumptions:
\begin{equation}\label{e7f}
\begin{aligned}
\text{(i)}\,\,\, & (V, \omega, \mu) \text{ is a connected, locally finite, undirected, weighted graph};\\
\text{(ii)}\,\, \, &  \text{ there exists a constant } C>0 \text{ such that for every } x\in V, \\ 
   & \sum_{y\sim x}\omega_{xy}\leq C \mu(x);\\
\text{(iii)} \,\,\,& \text{there exists a \textit{pseudo metric}}\,\, d \,\,\,\text{such that the jump size $j$ is finite}; \\
\text{(iv)}\,\,\,& \text{the ball}\,\,\, B_r(x) \,\,\,\text{with respect to}\,\,\, d\,\,\, \text{is a finite set, for any}\,\,\, x\in V,\,\,\, r>0;\\
\text{(v)}\,\, & \text{ for some } x_0\in V, R_0>1, \alpha\in [0, 1], C>0, \text{ there holds } \\
& \Delta d(x, x_0)\leq \frac C{d^\alpha(x, x_0)} \text{ for any }\, x\in V\setminus B_{R_0}(x_0).
\end{aligned}
\end{equation}

\begin{remark}
    We explicitly note that condition \eqref{e7f}-(v) is always satisfied with $\alpha=0$ if \eqref{e7f}-(ii) and \eqref{e7f}-(iii) hold. Indeed, for every choice $x_0\in V$ we have
    \begin{equation*}
\Delta d(x,x_0)=\frac{1}{\mu(x)}\sum_{y\sim x}\omega_{xy}[d(y,x_0)-d(x,x_0)]\leq \frac{1}{\mu(x)}\sum_{y\sim x}\omega_{xy} d(x,y)\leq  Cj, \quad \text{ for every } x\in V.
\end{equation*}
Moreover we also note that if there exist $C>0$, $x_0\in V$, $p>1$ and $R_0>1$ such that
$$\Delta d^p(x,x_0)\leq C$$
for every $x\in B_{R_0}^c(x_0)$, then \eqref{e7f}-(v) holds for $\alpha=\min\{p-1,1\}$. Indeed by the convexity of the real valued function $h(t)=t^p$ and $x\in B_{R_0}^c(x_0)$, we get
\begin{equation*}
\begin{split}
    C&\geq\Delta d^p(x,x_0)=\frac{1}{\mu(x)}\sum_{y\sim x}\omega_{xy} (d^p(y,x_0) - d^p(x,x_0))\\
    &\geq \frac{p}{\mu(x)}\sum_{y\sim x} \omega_{xy} d^{p-1}(x,x_0)(d(y,x_0) - d(x,x_0))\\
    &=pd^{p-1}(x,x_0)\Delta d(x,x_0).
\end{split}
\end{equation*}
\end{remark}

\begin{remark}\label{curvatura}
Observe that, on Riemannian manifolds, assumption \eqref{e7f}-(v) is satisfied if a suitable curvature condition holds. In fact, if the Ricci curvature fulfills the bound 
$$\operatorname{Ric}(x)\geq - C d^{-2\alpha}(x, x_0)\quad \text{ whenever }\; d(x, x_0)>1,$$
for some $C>0$, $\alpha\in [0, 1]$ and a fixed reference point $x_0$, then, by means of standard Laplacian comparison theorems, \eqref{e7f}-(v) is satisfied.
\end{remark}

%-------------------------------------------------------%
%                                                       %
% 						MAIN RESULT 					%
%                                                       %
%-------------------------------------------------------%

\section{Main results}\label{mr}

\begin{theorem}\label{t: mainthm}
Let assumption \eqref{e7f} be satisfied. Let $v\colon V\to \R$ be a positive function, and $\sigma>1.$ Suppose that

\begin{equation}\label{i: volumegrowth} 
        \sum_{x\in B_{2R}(x_0)\setminus B_{R}(x_0)} v^{-\frac{1}{\sigma-1}}(x)\mu(x)\leq CR^{\frac{(1+\alpha)\sigma}{\sigma-1}}\quad \text{ for all }
     R\geq R_0\,,
\end{equation}
with $\alpha\in[0,1]$, $x_0\in V$ as in \eqref{e7f}-(v). Let $u:V\to \R$ be a non-negative solution of \eqref{e1p}.
Then $u\equiv 0$.
\end{theorem}

As a simple consequence of Theorem \ref{t: mainthm}, when $v\equiv 1$ on $V$ we obtain the following result.
\begin{corollary} Let assumption \eqref{e7f} be satsfied, let $\sigma>1$. Suppose that 
\begin{equation}        
\Vol(B_{2R}(x_0)\setminus B_R(x_0))\leq CR^{\frac{(1+\alpha)\sigma}{\sigma-1}}\quad \text{ for every } R\geq R_0\,,
\end{equation}
with $\alpha\in[0,1]$, $x_0\in V$ as in \eqref{e7f}-(v). Let $u:V\to \R$ be a non-negative solution of $$\Delta u +  u^\sigma\leq 0 \quad \text{ in } \; V\,.$$
Then $u\equiv 0$.
\end{corollary}

\begin{remark}
\noindent (i) In Example \ref{ex1} we apply Theorem \ref{t: mainthm} on the \textit{N-dimensional integer lattice graph}  $(\mathbb Z^N, \omega, \mu)$, with a suitable $\mu$, for $v\equiv1$ and $$1<\sigma\leq\frac{N}{N-2}\text{ if }N>2,\,\,\,\,\,\sigma>1\text{ if }N=1,2.$$ Furthermore in Example \ref{ex2} we apply Theorem \ref{t: mainthm} on a tree where condition \eqref{e2p} fails and $v\equiv1$; therefore, the nonexistence result in \cite{GHS} cannot be applied. 

\noindent (ii) Observe that condition \eqref{i: volumegrowth} is optimal. In fact, in view of Theorem \ref{teooptimal} below there exists a positive solution $u$ to \eqref{e1p} with $v\equiv 1$ on $(\mathbb{Z}^N, \omega, \mu)$, for the same $\mu$ appearing in Example \ref{ex1}  and with $$N>2 \text{ and }\sigma>\frac{N}{N-2},$$ where condition \eqref{e7f} holds while \eqref{i: volumegrowth} fails. 

In particular, the critical exponent for \eqref{e1p} with $v\equiv1$ on $(\mathbb{Z}^N, \omega, \mu)$ for $N>3$ is $\sigma^*=\frac{N}{N-2}$, which is the same critical exponent as that of the problem $$\begin{cases}\Delta u+u^\sigma\leq0&\quad\text{ on }\mathbb{R}^N,\\u\geq0&\quad\text{ on }\mathbb{R}^N.\end{cases}$$

\noindent (iii) Also Theorem \ref{teo2optilmal} is about optimality of condition \eqref{i: volumegrowth}. Indeed, it guarantees the existence of a positive solution $u$ to \eqref{e1p} with $v\equiv 1$ on a {\it homogeneous tree of degree $N$} \, $$(\mathbb {T}_N, E, \mu),$$ with a suitable $\mu$, where condition \eqref{e7f} holds and \eqref{i: volumegrowth} is not fulfilled.
\end{remark}

\section{Proof of Theorem \ref{t: mainthm}}\label{proofmt}
We begin with a strong maximum principle for nonnegative superharmonic functions for the weighted Laplace operator on $(V,\omega,\mu)$.
\begin{lemma}
Let $u\colon V \to \R$ be a function fulfilling $$u\geq 0 \quad \text{ in } V,$$ and 
$$\Delta u\leq 0\quad \text{ in } V\,.$$
Then either $$u>0\quad \text{ in } V,$$ or 
$$u\equiv 0 \quad \text{ in } V.$$ 
\end{lemma}

\begin{proof}
Suppose that $u(x)=0$ for some $x\in V$. We have
\begin{equation*}
 0\geq \Delta u(x)= \frac{1}{\mu(x)}\sum_{y\sim x} \omega_{x y} (u(y)-u(x))= \frac{1}{\mu(x)}\sum_{y \sim x} \omega_{x y} u(y)\geq 0.
\end{equation*}
Therefore $u(y)=0$ for every $y\in V$ such that $y\sim x$. Since the graph is connected, we get $u\equiv 0$.
\end{proof}

\begin{proof}[Proof of Theorem \ref{t: mainthm}]
    Let $\psi \in C^{2}([0,+\infty))$ be any cut-off function on $[0,+\infty)$ which satisfies the following conditions
    \begin{itemize}
        \item $\psi \equiv 0$ on $[2,+\infty)$ and $\psi\equiv 1$ on $[0,1]$.
        \item $\psi'(x)\leq 0$, for each $x\in [0,+\infty)$. \item There exists $C>0$ such that $|\psi'(x)|\leq C$ and $|\psi''(x)|\leq C$, for each $x\in [0,+\infty)$. 
    \end{itemize}
Fix $R\geq \max\{R_0,j\}$. We define $$\p: V\to \R$$ by $$\p(x)\coloneqq \psi\left(\frac{d(x)}{R}\right),$$ where $d(x)\coloneqq d(x,x_0)$. Similarly, for $r>0$, we define $B_r\coloneqq B_r(x_0)$.

\smallskip

We claim that there exists a constant $C>0$ such that, for each $x\in V$ the following estimate holds
\begin{equation}\label{eq: laplacephiestimate}
    -\Delta \varphi(x)\leq \frac{C}{R^{1+\alpha}}\bone_{B_{2R+j}\setminus B_{R-j}}.
\end{equation}
Indeed, observing that for each $t,r\geq 0$ the formula
\begin{equation*}
    \psi(t)=\psi(r) + \psi'(r)(t-r) + \frac{\psi''(\xi)}{2} (t-r)^2
\end{equation*}
holds for some $\xi\geq 0$ between $r$ and $t$, we get
\begin{equation*}
\begin{split}
    -\Delta \p(x)&= -\frac{1}{\mu(x)}\sum_{y\sim x}\omega_{xy}(\p(y)-\p(x))\\
    &= -\frac{1}{\mu(x)} \sum_{y\sim x} \omega_{xy} \left(\psi\left(\frac{d(y)}{R}\right) - \psi\left(\frac{d(x)}{R}\right)\right)\\
    &= -\frac{1}{\mu(x)} \sum_{y\sim x}\omega_{xy} \left(\psi'\left(\frac{d(x)}{R}\right)\left(\frac{d(y)-d(x)}{R}\right)+ \frac{\psi''(\xi)}{2}\left(\frac{d(y)-d(x)}{R}\right)^2\right),
        \end{split}
    \end{equation*}
    for some $\xi>0$ between $\frac{d(x)}{R}$ and $\frac{d(y)}{R}$. Therefore
    \begin{equation*}
\begin{split}
     -\Delta \p(x)&= -\psi'\left(\frac{d(x)}{R}\right)\left(\frac{1}{\mu(x)} \sum_{y\sim x}\omega_{xy} \frac{d(y)-d(x)}{R} \right) - \frac{1}{\mu(x)} \sum_{y\sim x}\omega_{xy} \frac{\psi''(\xi)}{2} \frac{(d(y)-d(x))^2}{R^2}\\
    &= -\psi'\left(\frac{d(x)}{R}\right) \frac{\Delta d(x)}{R} -\frac{1}{\mu(x)} \sum_{y\sim x}\omega_{xy} \frac{\psi''(\xi)}{2} \frac{(d(y)-d(x))^2}{R^2}\\
    &\leq  \left|\psi'\left(\frac{d(x)}{R}\right)\right| \frac{\Delta d(x)}{R} + \frac{1}{\mu(x)} \sum_{y\sim x}\omega_{xy} \left|\frac{\psi''(\xi)}{2}\right| \frac{(d(y)-d(x))^2}{R^2}\\
    &\leq C \left(\frac{1}{R^{1+\alpha}}\bone_{B_{2R}\setminus B_{R}}+\frac{1}{R^2}\bone_{B_{2R+j}\setminus B_{R-j}}\right),
\end{split}    
\end{equation*}
where in the last inequality we have used \eqref{e7f}-(v), and the following facts: $$|\psi'|\leq C,|\psi''|\leq C,$$ 
$$\psi'\equiv 0, \; \psi''\equiv 0 \; \text{ outside } [1,2],$$
$$|d(y)-d(x)|\leq d(x,y)\leq j, \text{ for every } x,y\in V \text{ such that } x\sim y.$$ Finally, observing that $\alpha\in[0,1]$ and $R>1$, we get
\begin{equation*}
    C\left(\frac{1}{R^{1+\alpha}}\bone_{B_{2R}\setminus B_{R}}+\frac{1}{R^2}\bone_{B_{2R+j}\setminus B_{R-j}}\right)\leq \frac{C}{R^{\alpha+1}} \bone_{B_{2R+j}\setminus B_{R-j}}
\end{equation*}
and the claim is proved. For simplicity, from now on we will denote the subset $B_{2R+j}\setminus B_{R-j}\subset V$ by $A_R$.

Now we are going to show that there exists a constant $C$ such that
\begin{equation}\label{eq: integrability}
    \sum_{x\in V}\mu(x)u^{\sigma}(x)v(x) \leq C.
\end{equation}
Since $u$ fulfills \eqref{e1p}, we have for each $x\in V$
\begin{equation*}
    \mu(x) v(x)u^{\sigma}(x)\leq -\mu(x) \Delta u(x). 
\end{equation*}
Let $s>\frac{\sigma}{\sigma - 1}$. Multiplying both terms by $\p^s(x)$ and summing over all $x\in V$, we get
\begin{equation*}
\begin{split}
    \sum_{x\in V}\mu(x)v(x)\p^s(x)u^\sigma(x) & \leq -\sum_{x\in V} \p^s(x) \mu(x)\Delta u(x)
    =-\sum_{x\in V}\sum_{y\sim x}\p^s(x) \omega_{xy}\nabla_{xy}u\\
    &=-\sum_{x\in V}\sum_{y\sim x}\omega_{xy} u(x)\nabla_{xy}\p^s
    \leq -s\sum_{x\in V} \sum_{y\sim x} \omega_{xy} u(x)\p^{s-1}(x)\nabla_{xy}\p,
\end{split}
\end{equation*}
where in the last inequality we have used the convexity of the function $h(t)=t^s$, i.e. $$h(t_1)\geq h(t_0) + h'(t_0)(t_1 - t_0)\;\, \text{ for every } t_0,t_1 \in \R^+.$$ Indeed,
\begin{equation*}
\nabla_{xy}\p^s=\p^s(y)-\p^s(x)\geq s\p^{s-1}(x)(\p(y)-\p(x))=s\p^{s-1}(x)\nabla_{xy}\p.
\end{equation*}
Combining the definition of the graph Laplacian with \eqref{eq: laplacephiestimate} we get
\begin{equation}\label{145}
\begin{split}
   \sum_{x\in V}\mu(x)v(x)\p^s(x)u^\sigma(x) & \leq -s\sum_{x\in V} \sum_{y\sim x} \omega_{xy} u(x)\p^{s-1}(x)\nabla_{xy}\p\\
   &=-s\sum_{x\in V} u(x)\p^{s-1}(x)\sum_{y\sim x} \omega_{xy}\nabla_{xy}\p\\
   &=-s \sum_{x\in V} u(x)\p^{s-1}(x)\mu(x)\Delta\p\\
   &\leq \frac{C}{R^{\alpha+1}}\sum_{x\in A_R}\mu(x)u(x)\p^{s-1}(x).
\end{split}
\end{equation}
By a standard application of Young's inequality the last term of the above inequalities is less or equal than
\begin{equation*}
    \frac{1}{\sigma}\sum_{x\in A_R}\mu(x)u^{\sigma}(x)\p^{s}(x)v(x)+\frac{\sigma-1}{\sigma}\frac{C}{R^{(1+\alpha)\frac{\sigma}{\sigma-1}}}\sum_{x\in A_R}\mu(x)\p^{s-\frac{\sigma}{\sigma -1}}(x)v^{-\frac{1}{\sigma -1}}(x).
\end{equation*}
Hence, by taking the first and the last term of the above inequalities and observing that $$0\leq\p^{s-\frac{\sigma}{\sigma -1}}(x)\leq 1,$$ we obtain
\begin{equation*}
    \sum_{x\in V}\mu(x)u^{\sigma}(x)v(x)\p^s(x)\leq \frac{C}{R^{(1+\alpha)\frac{\sigma}{\sigma-1}}}\sum_{x\in A_R}\mu(x)v^{-\frac{1}{\sigma -1}}(x).
\end{equation*}
Finally, observing that 
$$ \sum_{x\in B_{R}}\mu(x)u^{\sigma}(x)v(x)\leq\sum_{x\in V}\mu(x)u^{\sigma}(x)v(x)\p^s(x)$$ 
and that, for $R$ large enough, $A_R\subset B_{4R}\setminus B_{R/2}$,
from \eqref{i: volumegrowth} we get
\begin{equation*}
\begin{split}
    \sum_{x\in B_{R}}\mu(x)u^{\sigma}(x)v(x)&\leq \frac{C}{R^{(1+\alpha)\frac{\sigma}{\sigma-1}}}\sum_{x\in A_R}\mu(x)v(x)^{-\frac{1}{\sigma -1}}\\
    &\leq \frac{C}{R^{(1+\alpha)\frac{\sigma}{\sigma-1}}}\bigg[\sum_{x\in B_{4R}\setminus B_{2R}} \mu(x)v(x)^{-\frac{1}{\sigma -1}}+\sum_{x\in B_{2R}\setminus B_{R}} \mu(x)v(x)^{-\frac{1}{\sigma -1}}\\
    &\hspace{3cm}+\sum_{x\in B_R\setminus B_{R/2}} \mu(x)v(x)^{-\frac{1}{\sigma -1}}\bigg]\leq C.
\end{split}
\end{equation*}
Since this last estimate is independent of $R$, we obtain \eqref{eq: integrability}. We are now in the position of proving that $u\equiv 0$. By \eqref{145} we have
\begin{equation*}
     \sum_{x\in V}\mu(x)v(x)\p^s(x)u^\sigma(x)\leq \frac{C}{R^{\alpha+1}}\sum_{x\in A_R}\mu(x)u(x)\p^{s-1}(x).
\end{equation*}
By applying the H\"older inequality to the right hand side, then by applying \eqref{i: volumegrowth} we obtain
\begin{equation*}
\begin{split}
\sum_{x\in V}\mu(x)v(x)\p^s(x)u^\sigma(x)&\leq \frac{C}{R^{1+\alpha}}\left(\sum_{x\in A_R}\mu(x)u^{\sigma}(x)\p^{s}(x)v(x)\right)^{\frac{1}{\sigma}}\left(\sum_{x\in A_R}\mu(x)\p^{s-\frac{\sigma}{\sigma -1}}(x)v^{-\frac{1}{\sigma -1}}(x)\right)^{\frac{\sigma-1}{\sigma}}\\
&\leq \frac{C}{R^{1+\alpha}} \left(\sum_{x\in A_R}\mu(x)u^{\sigma}(x)v(x)\right)^{\frac{1}{\sigma}} \left(\sum_{x\in A_R}\mu(x)v^{-\frac{1}{\sigma -1}}(x)\right)^{\frac{\sigma-1}{\sigma}}\\
&\leq C\left(\sum_{x\in A_R}\mu(x)u^{\sigma}(x)v(x)\right)^{\frac{1}{\sigma}}.
\end{split}
\end{equation*}
By \eqref{eq: integrability}, the right hand side goes to zero as $R$ goes to infinity. It follows that
\begin{equation*}
    \sum_{x\in V}\mu(x)v(x)u^\sigma(x)\leq 0.
\end{equation*}
Since $\mu$ and $v$ are positive, we get $u\equiv 0$ in $V$.
\end{proof}

%-------------------------------------------------------%
%                                                       %
% 						EXAMPLES 					    %
%                                                       %
%-------------------------------------------------------%

\section{Examples}\label{ex}
\begin{example}\label{ex1}
The \textit{N-dimensional integer lattice graph} is the graph consisting of the set of vertices $\mathbb{Z}^N$, the edge weight defined by
\[\omega:\mathbb Z^N\times \mathbb Z^N\to [0,+\infty),\]
\begin{equation*}
    \omega_{xy}= \begin{cases}
1, \,\, \text{ if } \sum_{k=1}^N|y_k-x_k|=1,\\
0, \,\, \text{ otherwise.} 
\end{cases}
\end{equation*}
and the node measure $$\mu(x)\coloneqq \sum_{y\sim x}\omega_{xy}=2N.$$

In other words,  the set of edges is
\begin{equation*}
    E=\left\{(x,y)\colon x,y \in \mathbb{Z}^N,\,\, \sum_{i=1}^{N}|x_i-y_i|=1\right\};
\end{equation*}
moreover, $y\sim x$ if and only if there exists $k\in\{1,\dots, N\}$ such that $x_k=y_k\pm 1$ and $x_i=y_i$ for $i\neq k$. Therefore, 
\begin{equation*}
    \omega_{xy}= \begin{cases}
1, \,\, \text{ if } y\sim x,\\
0, \,\, \text{ if } y\not\sim x.
\end{cases}
\end{equation*}

  Then for every function $u:\mathbb{Z}^N\rightarrow\mathbb{R}$ we have $$\Delta u(x)=\frac{1}{2N}\sum_{y\sim x}(u(y)-u(x))\qquad \text{ for all }x\in\mathbb{Z}^N.$$
We see that the only nonnegative solution to 
$$\Delta u + u^{\sigma}\leq 0\quad \text{ in }\; \mathbb{Z}^N\,,$$ 
is $u\equiv 0$, provided that 
\begin{itemize}
        \item  $N\in\{1,2\}$ and $\sigma>1$, or
        \item $N> 2$ and $1<\sigma \leq \frac{N}{N-2}$.
    \end{itemize}
    
In fact, we consider the graph $(\mathbb{Z}^N,E, \mu)$ endowed with the euclidean distance 
\begin{equation*}
|x-y|\coloneqq d(x,y)=\left(\sum_{i=1}^N(x_i-y_i)^2\right)^{1/2}.
\end{equation*}
    For $R>1$, we have
    \begin{equation*}
        \sum_{x\in B_R(0)}\mu(x)= 2N\sum_{x\in B_{R}(0)}1\leq C R^{N}\leq C R^{\siii},
    \end{equation*}
where in the last inequality we have used that $1<\sigma \leq\frac{N}{N-2}$, which implies $N\leq \siii$, in the case $N>2$, while $\siii>2\geq N$ if $N\in\{1 , 2 \}$. Condition \eqref{i: volumegrowth} of Theorem \ref{t: mainthm}, with $v\equiv1$ on $\mathbb{Z}^N$, is satisfied for $\alpha=1$. It remains to show that condition \eqref{e7f}-(v) of Theorem \ref{t: mainthm} holds as well. For $x\sim y$ we have $x=(x_1,\dots, x_N)$ and $y=(x_1,\dots, x_k\pm 1,\dots x_N)$ for some $k\in \{1,\dots, N\}$. In this case we obtain
\begin{equation*}
    d^2(y,0)-d^2(x,0)=(|x|^2\pm 2x_k + 1) - |x|^2 = \pm 2x_k +1,
\end{equation*}
which implies that
\begin{equation*}
\begin{split}
    \Delta d^2(x,0) &= \sum_{y\sim x}\frac{\omega_{xy}}{\mu(x)}(d^2(y,0)-d^2(x,0))\\
    &= \frac{1}{2N} \sum_{k=1}^N(2x_k+1) + (-2 x_k+1)=1.
\end{split}
\end{equation*}
By convexity of the real valued function $h(t)=t^2$, we get
\begin{equation*}
\begin{split}
    1&=\Delta d^2(x,0)=\frac{1}{2n}\sum_{y\sim x} (d^2(y,0) - d^2(x,0))\\
    &\geq \frac{1}{2n}\sum_{y\sim x} 2 d(x,0)(d(y,0) - d(x,0))\\
    &=2d(x,0)\Delta d(x,0).
\end{split}
\end{equation*}
From the previous inequality we get
\begin{equation*}
    \Delta d(x,0) \leq \frac{1}{2 d(x,0)}.
\end{equation*}
Finally, applying Theorem \ref{t: mainthm} we obtain the assertion.
\end{example}
\medskip

The next example of this section shows that Theorem \ref{t: mainthm} can be applied to weighted graphs which lack property \eqref{e2p}.

We recall that a graph $(\mathbb{T}, E, \mu)$ is a tree if for each $x,y\in \mathbb{T}$ there exists only one (minimal) path that connects $x$ to $y$. Here with an abuse of notation we declare the set of edges $E$ of the graph, i.e. we declare which $x,y\in \mathbb{T}$ satisfy $y\sim x$, instead of a weight function $\omega$; we tacitly assume that we can consider a suitable weight function $\omega$ on the graph, which induces the set of edges $E$.

We consider on $\mathbb{T}$ the distance between two vertices $x,y\in \mathbb{T}$, denoted by $d(x,y)$, as the cardinality of the set of the edges of the unique (minimal) path between them. Note that the distance $d$ has jump size $j=1$. For $n\geq 0$, we set $$D_n\coloneqq \{x\in \mathbb{T}\colon d(0,x)=n\},$$ where $0$ is the root of the tree $\mathbb{T}$, while $E_n$ denotes the collection of all edges from vertices in $D_n$ to vertices in $D_{n+1}$. Note that if $x\in D_n$ for some $n\geq1$, then there exists a unique $y\in D_{n-1}$ such that $y\sim x$.
A tree is said \textit{homogeneous of degree $N>1$} if each element has degree $N$, and it will be denoted by $(\mathbb{T}_N, E, \mu)$.

%In order to simplify the computation, we will assume that $0$ has degree $N-1$.

\begin{example}\label{ex2}
Let $(\mathbb{T},E, \mu)$ be the tree with root $0$ such that each element of $D_n$ has $n+1$ edges. Note that the cardinality of $D_n$ is equal to $(n-1)!$ for every $n\geq 1$. 
For $x,y \in \mathbb{T}$ we define $$\omega_{xy}=\frac{1}{\min\{h!,k!\}}, \text{ if } y\sim x \text{ and } x\in D_h \text{ and } y\in D_k,$$ (note that if $y\sim x$, then $|h-k|=1$), whereas $$\omega_{xy}=0 \text{ if } x\not\sim y.$$ 
Hence, if $x\in D_n$ and $n\geq 1$, then there are $y_1,\dots ,y_n \in D_{n+1}$ such that $y_k\sim x$ for $k=1,\dots, n$ and only one $y_0 \in D_{n-1}$ such that $y_0\sim x$. Therefore we have $$\omega_{xy_k}=\frac{1}{n!} \text{ for } k\in \{1,\dots, n\},$$ and $$\omega_{xy_0}=\frac{1}{(n-1)!}.$$ We define the map $\mu\colon \mathbb{T}\to [0,+\infty)$ by
\begin{equation*}
    \mu(x)\coloneqq\sum_{y\sim x}\omega_{xy}= \frac{n}{n!} + \frac{1}{(n-1)!}=\frac{2}{(n-1)!} \quad \text{for } x\in D_n, \, n\geq 1.
\end{equation*}
In particular condition \eqref{e2p} is not satisfied. Indeed, for any $x\in D_n$ with $n\geq1$ we have $$\frac{\omega_{xy_k}}{\mu(x)}=\frac{1}{2n} \quad \text{ for } \; k\in\{1,\dots,n\},$$ and $$\frac{\omega_{xy_0}}{\mu(x)}=\frac 12.$$ Finally, if $x=0\in D_0$, then there exists a unique $y\in \mathbb{T}$ such that $y\sim x$,  thus we have $\omega_{0y}=1$ and $\mu(0)=1$.

It remains to show that conditions \eqref{i: volumegrowth} and \eqref{e7f}-(v) of Theorem \ref{t: mainthm} are satisfied. In order to do that let $x\in D_n$, with $n\geq 1$
\begin{equation*}
\begin{split}
    \Delta d(x,0)&=\sum_{y\sim x}\frac{\omega_{xy}}{\mu(x)} [d(y,0) - d(x,0)]\\
    &=\sum_{k=1}^n\frac{1}{2n}[d(y_k,0) - d(x,0)] + \frac{1}{2} [d(y_0,0) - d(x,0)]\\
    &= \frac 12 -\frac 12=0.
\end{split}
\end{equation*}
Therefore, in particular it holds that
$$\Delta d(x,0)\leq \frac{C}{d(x,0)},$$
and condition \eqref{e7f}-(v) of Theorem \ref{t: mainthm} is satisfied for $\alpha=1$. We conclude by observing that condition \eqref{i: volumegrowth} is satisfied for the same value of $\alpha$. Indeed, let $N\in \N$ and $N\geq 1$, we have \begin{equation*}
    \begin{split}
\operatorname{Vol}(B_N)&= \sum_{x\in B_N}\mu(x)=\sum_{k=0}^{N}\sum_{x\in D_k}\mu(x)\\
&=1+ \sum_{k=1}^{N}\frac{2}{(k-1)!} (k-1)!\\
&= 2N +1.
    \end{split}
\end{equation*}
Finally, let $R\in\R$ and $R\geq 2$. We denote by $[R]\coloneqq \max\{x\in \N\colon x\leq R\}$ the \textit{integer part of $R$}. Then we have 
\begin{equation*}
\operatorname{Vol}(B_{[R]})\leq \operatorname{Vol}(B_R)\leq \operatorname{Vol}(B_{[R]+1}),
\end{equation*}
which readily implies that $\operatorname{Vol}(B_R)\asymp R$ and $\operatorname{Vol}(B_{2R}\setminus B_R)\asymp R$.

Thus, for $v\equiv 1$ and $\sigma>1$, by Theorem \ref{t: mainthm} the inequality $$\Delta u +v u^{\sigma}\leq 0 \quad \text{in }\; \mathbb{T}$$ has no nontrivial nonnegative solutions.
\end{example}

\section{Optimality of the volume growth condition}\label{optimal}

The next result, in conjunction with Example \ref{ex1}, shows that Theorem \ref{t: mainthm} is sharp with respect to $\sigma$. 
\begin{theorem}\label{teooptimal} 
Let $(\mathbb{Z}^N,E, \mu)$ be as in Example \ref{ex1}. 
    Let $N>2$, $\e>0$, and $\sigma=\frac{N}{N-2}+\e$. Then there exists a non-trivial positive solution $u$ of $$\Delta u + u^\sigma \leq 0 \quad \text{ in }\;\, \mathbb{Z}^N.$$ 
\end{theorem}
\begin{proof}
    Let $\varphi\colon \R^+\to \R^+$ be defined by
    \begin{equation*}
        \varphi(t)=\frac{\delta}{(K+t)^\gamma}
    \end{equation*}
    where $\gamma=\frac{1}{\sigma-1}$ and the parameters $\delta >0$ and $K>0$ will be chosen small and large enough respectively. Let $u\colon \mathbb{Z}^N\to \R$ be defined by
    \begin{equation*}
    u(x)=\varphi(|x|^2)=\frac{\delta}{(K+|x|^2)^{\gamma}},
    \end{equation*}
   where $\displaystyle{|x|^2\coloneqq \sum_{i=1}^N x_i^2.}$ By the Taylor expansion formula, 
$$ \p(t)= \p(r) + \p'(r)(t-r) + \frac{\p''(r)}{2}(t-r)^2 + \frac{\p'''(\eta)}{6}(t-r)^3, $$
 with $\eta$ between $t$ and $r$.  Then, we obtain
\begin{equation}\label{eq: sviluppoDeltau}
    \begin{split}
    \Delta u(x) = \frac{1}{2N} \sum_{y\sim x} &\left[-\frac{\gamma \delta}{(K+|x|^2)^{\gamma+1}} (|y|^2-|x|^2) + \frac{\gamma(\gamma+1) \delta}{2(K+|x|^2)^{\gamma+2}} (|y|^2-|x|^2)^2\right.\\
    &\left.\,\, - \frac{\gamma(\gamma+1)(\gamma+2) \delta}{6(K+\eta)^{\gamma+3}} (|y|^2-|x|^2)^3 \right],
    \end{split}
\end{equation}
  with $\eta$ between $|y|^2$ and $|x|^2$. Now we recall that if $y\sim x$ we have $x=(x_1,\dots, x_N)$ and $y=(x_1,\dots, x_k\pm 1,\dots x_N)$ for some $k\in \{1,\dots, N\}$. Then we obtain
\begin{equation*}
    |y|^2-|x|^2=(|x|^2\pm 2x_k + 1) - |x|^2 = \pm 2x_k +1.
\end{equation*}
Thus
\begin{equation*}
\begin{split}
\sum_{y\sim x}(|y|^2-|x|^2)&=2N, \\
\sum_{y\sim x}(|y|^2-|x|^2)^2&= 8 |x|^2 +2N,\\
| (|y|^2-|x|^2)^3|&\leq C(|x|^3+1),
\end{split}
\end{equation*}
for some $C>0$. By \eqref{eq: sviluppoDeltau}, 
\begin{equation*}
\begin{split}
    \Delta u(x) + u^{\sigma}(x) &=-\frac{\gamma \delta}{(K+|x|^2)^{\gamma+1}}  + \frac{\gamma(\gamma+1) \delta}{(K+|x|^2)^{\gamma+2}} \frac{4|x|^2 + N}{2N}\\
    & \,\,\,\,\;\, -\frac{1}{2N}\sum_{y\sim x} \frac{\gamma(\gamma+1)(\gamma+2) \delta}{6(K+\eta)^{\gamma+3}} (|y|^2-|x|^2)^3 + \frac{\delta^\sigma}{(K+|x|^2)^{\gamma\sigma}}\\
    &= \frac{\delta}{(K+|x|^2)^{\gamma+2}}\bigg[-\gamma (K+|x|^2) +\gamma(\gamma+1)\frac{4|x|^2 + N}{2N} \\
    &\hspace{3,3cm} +\delta^{\sigma-1}(K+|x|^2) + \theta(x) \bigg]\\
    &= \frac{\delta}{(K+|x|^2)^{\gamma+2}}\bigg[-\left(\gamma - \frac{2\gamma(\gamma+1)}{N} - \delta^{\sigma-1} \right)|x|^2\\
    &\hspace{3,3cm} - \left(\gamma K -\frac{\gamma(\gamma+1)}{2} - \delta^{\sigma-1}K\right) + \theta(x) \bigg].
\end{split}
\end{equation*}
where we have used that $\gamma \sigma=\gamma+1$ and where
\begin{equation*}
    \theta(x) = -\frac{(K+|x|^2)^{\gamma+2}}{12N}\sum_{y\sim x} \frac{\gamma(\gamma+1)(\gamma+2)}{(K+\eta)^{\gamma+3}} (|y|^2-|x|^2)^3.
\end{equation*}
We set $2\lambda \coloneqq \gamma - \frac{2\gamma(\gamma+1)}{N} $. Then, since $\gamma=\frac{1}{\sigma-1}$ and $\sigma=\frac{N}{N-2} +\varepsilon$, we get $\lambda>0$ and by choosing $0<\delta \leq \min\{\lambda^{\frac{1}{\sigma-1}}, (\frac{\gamma}{2})^{\frac{1}{\sigma-1}}\}$, we obtain
\begin{equation*}
    \begin{split}
      \Delta u(x) + u^{\sigma}(x)&\leq  \frac{\delta}{(K+|x|^2)^{\gamma+2}}\bigg[-\lambda|x|^2- \left(\frac{\gamma}{2}K -\frac{\gamma(\gamma+1)}{2} \right) + \theta(x) \bigg]\\
      &\leq -\frac{\delta}{(K+|x|^2)^{\gamma+2}}[\lambda|x|^2+C_0 - \theta(x) ],
    \end{split}
\end{equation*}
with $C_0>0$ arbitrarily large, up to choosing $K\geq \frac{2C_0}{\gamma} + \gamma+1$. We now estimate $\theta(x)$. For $y\sim x$, by construction we have 
$$ \min\{|x|^2,|y|^2\}\leq \eta\leq \max\{|x|^2,|y|^2\}.$$
Since $|x|^2\in\mathbb{N}$, for $|x|\geq1$ we have $|y|\geq|x|-1$ and hence $\eta\geq (|x|-1)^2\geq \frac 16 (|x|^2-1)$. Clearly also for $|x|=0$ one has $\eta\geq\frac 16 (|x|^2-1)$. Thus
\begin{equation*}
\begin{split}
    |\theta(x)|&\leq C (K+|x|^2)^{\gamma+2} \sum_{y\sim x} \frac{1}{(K+\eta)^{\gamma+3}} ||y|^2-|x|^2|^3\\
    &\leq C (K+|x|^2)^{\gamma+2} \frac{|x|^3+1}{(6K-1+|x|^2)^{\gamma+3}}\\
    & \leq C\frac{|x|^3 + 1}{K+|x|^2}\leq C(|x|+1).
\end{split}
\end{equation*}
Hence, 
\begin{equation*}
    \begin{split}
      \Delta u(x) + u^{\sigma}(x)&\leq  -\frac{\delta}{(K+|x|^2)^{\gamma+2}} \left[\lambda |x|^2 + C_0 - C(|x| +1)\right]\leq 0,
    \end{split}
\end{equation*}
up to choosing $C_0\geq 0$, and hence $K\geq 0$, large enough. Therefore $u$ is a positive solution of $$\Delta u + u^\sigma \leq 0 \quad \text{ in } \mathbb{Z}^N,$$ which concludes the proof.
\end{proof}

The next result is inspired by \cite[Theorem 4.1]{GHS}, and it shows that the volume growth assumption \eqref{i: volumegrowth} in Theorem \ref{t: mainthm} is sharp. 

\begin{theorem}\label{teo2optilmal}
Let  $\mathbb{T}_N$ be a homogeneous tree of degree $N\geq 2$. Let $\sigma>1$. For every $\e>0$, there exist $\omega\colon \mathbb{T}_N\times \mathbb{T}_N\to [0, +\infty)$ and $\mu\colon \mathbb{T}_N\to [0, +\infty)$ and a constant $C>0$ such that
\begin{romanenumerate}
    \item\label{i: sommaomega} for every $x\in \mathbb{T}_N$ 
    \begin{equation*}
    \sum_{y\sim x}\omega_{x y}= \mu(x);
    \end{equation*}
    \item\label{i: treevolumegrowth}  for every $n\in\N$
    \begin{equation*}
        \sum_{x\in B_{n}(0)} \mu(x)\asymp n^{\frac{2\sigma}{\sigma-1}+\e};
    \end{equation*}
    \item\label{i: treelaplacianestimate} for every $x\in \mathbb{T}_N$
    \begin{equation*}
        \Delta d(x,0)\leq \frac{C}{d(x,0)};
    \end{equation*}
    \item\label{i: treesolution} there exists a positive solution $u$ to $$\Delta u+u^{\sigma}\leq 0 \quad \text{ in } \; \mathbb{T}_N\,.$$
\end{romanenumerate}
\end{theorem}

\begin{proof}
    Let $n\in \N$ and $x\in D_n$, we define
    \begin{equation}\label{e: utree}
        u(x)=u_n=\frac{\delta}{(n+n_0)^\frac{2}{\sigma-1}}.
    \end{equation}
    For $n\in\N$ and $(x,y)\in E_n$, we define
    \begin{equation}\label{e: muxy}
\omega_{xy}=\omega_n=\frac{(n+n_0)^{\frac{\sigma+1}{\sigma-1}+\e}}{(N-1)^n},
    \end{equation}
 while if $y\not\sim x$ we define $\omega_{xy}=0$. Here $n_0 \in \N$ is large enough and $\delta>0$ is small enough (see the computation below). Defining $$\mu(x)\coloneqq \sum_{y\sim x}\omega_{xy},$$ \eqref{i: sommaomega} holds. Therefore, by \eqref{e: muxy}, for every $x\in D_n$, we  get
\begin{equation}
    \mu(x)=(N-1)\omega_n + \omega_{n-1}\qquad\text{for }n\geq1
\end{equation}
and $\mu(0)=N\omega_0$. Hence, by taking $n_1>1$ large enough, for every $x\in D_n$ and $n\geq n_1$ we have
\begin{equation*}
    \begin{split}
        \Delta d(x,0)&=\sum_{x\sim y}\frac{\omega_{xy}}{\mu(x)}[d(y,0) - d(x,0)]=\frac{(N-1)\omega_n - \omega_{n-1}}{(N-1)\omega_n + \omega_{n-1}}\\
        &=\frac{(n+n_0)^{\si+\e}-(n-1+n_0)^{\si+\e}}{(n+n_0)^{\si+\e}+(n-1+n_0)^{\si+\e}}=\frac{1 - (1-\frac{1}{n+n_0})^{\si+\e}}{1 +(1-\frac{1}{n+n_0})^{\si+\e}} \\
        &= \frac{1}{2+ o(1)}\left[\left(\si + \e\right)\frac{1}{n+n_0}+o\left(\frac{1}{n}\right)\right]\leq \frac{C}{n}.
    \end{split}
\end{equation*}
Thus \eqref{i: treelaplacianestimate} follows. Now we are going to prove \eqref{i: treevolumegrowth}. Let $n>n_1$
\begin{equation*}
\begin{split}
    \sum_{x\in B_n(0)}\mu(x)&=\sum_{i=0}^n\sum_{x\in D_i}\mu(x)\\    
    &=\sum_{i=1}^n N(N-1)^{i-1}[(N-1)\omega_i + \omega_{i-1}] + N\omega_0\\
    &=\sum_{i=1}^n N\left[(n_0+i)^{\si+\e} + (i+n_0-1)^{\si+\e}\right] + n_0^{\si+\e}N\\
    &\asymp n^{\siii+\e},
\end{split}
\end{equation*}
where in the last passage we have used that $n_1>>n_0$. It remains to show \eqref{i: treesolution}. Since the function $u:\mathbb{T}_N\to \R$ defined by \eqref{e: utree} is positive, it remains to show that it is a solution of $$\Delta u + u^\sigma\leq 0\quad \text{ in }\; \mathbb{T}_N,$$ if we choose $\delta>0$ small and $n_0>0$ large enough. Let $x\in D_n$; we consider two cases: \textbf{(a)} $n=0$, and \textbf{(b)} $n>0$.

\smallskip

\textbf{(a)} For $x=0$  the inequality $\Delta u(0) +u^\sigma(0)\leq 0$ is equivalent to $u_1-u_0+u_0^{\sigma}\leq 0$. Therefore we get
\begin{equation*}
    \frac{\delta}{(1+n_0)^{\sii}}-\frac{\delta}{n_0^{\sii}}+\frac{\delta^\sigma}{n_0^{\siii}}\leq 0.
\end{equation*}
This is equivalent to
\begin{equation*}
    \delta^{\sigma-1}\leq n_0^{\siii}\left(\frac{1}{n_0^{\sii}}-\frac{1}{(1+n_0)^{\sii}}\right)=n_0^2\left[\frac{(1+n_0)^{\sii}-n_0^{\sii}}{(1+n_0)^{\sii}}\right]\coloneqq\Lambda.
\end{equation*}
Notice that the right hand side is positive and independent on $n$.

\smallskip

\textbf{(b)} Similarly as the previous case, let $x\in D_n$. The inequality $\Delta u(x) +u^\sigma(x)\leq 0$ is equivalent to
\begin{equation*}
    \frac{(N-1)\omega_n u_{n+1}+ \omega_{n-1} u_{n-1}}{(N-1)\omega_n + \omega_{n-1}} - u_n + u_n^\sigma\leq 0.
\end{equation*}
This is equivalent to
\begin{equation*}
\frac{(n+n_0)^{\si+\e}\frac{\delta}{(n+n_0+1)^{\sii}} + (n+n_0 -1)^{\si+\e}\frac{\delta}{(n+n_0-1)^{\sii}}}{(n+n_0)^{\si+\e} + (n+n_0 -1)^{\si+\e}} - \frac{\delta}{(n+n_0)^{\sii}} + \frac{\delta^{\sigma}}{(n+n_0)^{\siii}} \leq 0.
\end{equation*}
In order to simplify the computation we will use the notation $k\coloneqq n+n_0$. Hence the previous formula is equivalent to 
\begin{equation*}
\frac{k^{\si+\e}\frac{\delta}{(k+1)^{\sii}} + (k -1)^{\si+\e}\frac{\delta}{(k-1)^{\sii}}}{k^{\si+\e} + (k -1)^{\si+\e}} - \frac{\delta}{k^{\sii}} + \frac{\delta^{\sigma}}{k^{\siii}} \leq 0.
\end{equation*}
By some standard computation, we get
\begin{equation*}
    \delta^{\sigma-1}\leq k^{\siii} \left[\frac{1}{k^{\sii}}-\frac{k^{\si+\e}(k-1)^{\sii}+(k+1)^{\sii}(k-1)^{\si+\e}}{[k^{\si+\e}+(k-1)^{\si+\e}](k+1)^{\sii}(k-1)^{\sii}}\right].
\end{equation*}
Then, we get
\begin{equation*}
\delta^{\sigma-1}\leq k^2\left[\frac{(1-\frac{1}{k})^{\sii}[(1+\frac{1}{k})^{\sii} - 1] + (1-\frac{1}{k})^{\si+\e} (1+\frac{1}{k})^{\sii}[(1-\frac{1}{k})^{\sii} - 1]}{[1+(1-\frac{1}{k})^{\si+\e}](1-\frac{1}{k})^{\sii}(1+\frac{1}{k})^{\sii}}\right].
\end{equation*}
The denominator converges to $2$, hence for $n_0$, and thus for $k$, large enough it can be bound above and below by positive constants. Applying the formula $$(1+x)^{\alpha}=1+\alpha x + \frac{\alpha(\alpha-1)}{2} x^2 +o(x^2)$$ several times to the numerator, we get 
\begin{equation*}
\begin{split}
    &\left(1-\frac{1}{k}\right)^{\sii}\left[\left(1+\frac{1}{k}\right)^{\sii} - 1\right] + \left(1-\frac{1}{k}\right)^{\si+\e} \left(1+\frac{1}{k}\right)^{\sii}\left[\left(1-\frac{1}{k}\right)^{\sii} - 1\right]\\
    &=\frac{1}{k^2}\sii\left[-\sii + \frac{1}{2}\left(\sii-1\right) + \si + \e -\sii +\frac{1}{2}\left(\sii - 1\right)\right]+o\left(\frac{1}{k^2}\right)\\
    &=\frac{1}{k^2}\sii\left[\frac{3-\sigma}{\sigma-1} - \frac{4}{\sigma-1} + \si +\e\right]+o\left(\frac{1}{k^2}\right)\\
    &=\frac{\e}{k^2}\sii+o\left(\frac{1}{k^2}\right).
\end{split}
\end{equation*}
Therefore for $n_0$, and thus for $k$, large enough we get
\begin{equation*}
     C_1\e\leq k^2\left[\frac{(1-\frac{1}{k})^{\sii}[(1+\frac{1}{k})^{\sii} - 1] + (1-\frac{1}{k})^{\si+\e} (1+\frac{1}{k})^{\sii}[(1-\frac{1}{k})^{\sii} - 1]}{[1+(1-\frac{1}{k})^{\si+\e}](1-\frac{1}{k})^{\sii}(1+\frac{1}{k})^{\sii}}\right]\leq C_2\e
\end{equation*}
for some constants $C_1,C_2>0$. Hence \eqref{i: treesolution} is satisfied by the function $u$ defined in \eqref{e: utree}, up to choosing $n_0$ large enough and $0<\delta^{\sigma-1}<\min\{\Lambda,C_1\e\}$.
\end{proof}

%-------------------------------------------------------%
%                                                       %
% 						BIBLIOGRAPHY    				%
%                                                       %
%-------------------------------------------------------%

\end{document}